\documentclass[11pt]{amsart}
\usepackage{amsfonts, amsmath, amssymb, amsthm, bbm, color, enumerate, graphicx, mathtools, tikz, relsize}
\usepackage{dsfont}
\usepackage{bbm}
\usepackage{marginnote}
\usepackage{color}
\usepackage{graphicx}
\usepackage[utf8]{inputenc}
\usepackage{enumitem}
\usepackage[hidelinks]{hyperref}

\newtheorem{theorem}{Theorem}[section]
\newtheorem{corollary}[theorem]{Corollary}
\newtheorem{lemma}[theorem]{Lemma}
\newtheorem{proposition}[theorem]{Proposition}

\newtheorem{definition}[theorem]{Definition}

\allowdisplaybreaks

\def\beq{ \begin{equation} }
\def\eeq{ \end{equation} }

\def\ep{\epsilon}

\def\square{\vcenter{\vbox{\hrule height .4pt
  \hbox{\vrule width .4pt height 5pt \kern 5pt
        \vrule width .4pt} \hrule height .4pt}}}

\def\P{\mathbb{P}}

\def\ZZ{\mathbb{Z}}
\def\FF{\mathcal{F}}

\def\EE{\mathcal{E}}

\def\HH{\mathcal{H}}
\def\AA{\mathcal{A}}

\def\ep{\varepsilon}

\definecolor{darkblue}{rgb}{0,0,0.6}

\definecolor{darkgreen}{rgb}{0,0.7,0}

\def\Unif{\mathrm{Unif}}
\def\ab{\mathrm{ab}}
\def\1{\mathds{1}}
\def\floor#1{\lfloor #1 \rfloor}

\numberwithin{equation}{section}

\def\mix{\mathrm{mix}}

\usepackage{bm}   
\def\Cl{\mathrm{Cl}}
\def\TV{\mathrm{TV}}

\def\id{\mathrm{id}}

\title[Conjugacy-invariant random walks on nilpotent groups]{Conjugacy-invariant random walks on nilpotent groups}

\author{Xiangying Huang}

\begin{document}
\maketitle
\begingroup
\renewcommand{\thefootnote}{}
\footnotetext{Department of Statistics and Operations Research, University of North Carolina at Chapel Hill, Chapel Hill, NC 27599, USA; Email: zoehuang@unc.edu.}
\endgroup
\begin{abstract}
We establish bounds on the mixing times of conjugacy-invariant random walks on finite nilpotent groups in terms of the mixing times of their projections onto the abelianization. This comparison framework shows that, in several natural cases of interest, the mixing behavior on a nilpotent group is governed by that of the projected walk on the abelianization, reducing the study of mixing to a simpler problem in the Abelian setting. As an application, these bounds yield cutoff for two examples of conjugacy-invariant walks on unit upper-triangular matrix groups previously studied by Arias-Castro, Diaconis, and Stanley (2004) and by Nestoridi (2019).
\end{abstract}

\section{Introduction}

We consider a continuous-time random walk $X=(X_t)_{t\ge 0}$ on a finite group $G$.
At each jump time, the walk updates its state by right-multiplying the current
position by an independent group element sampled from a probability distribution
$\mu$ on $G$, which we refer to as the \emph{jump distribution} of the walk.
The walk is said to be \emph{conjugacy invariant} if its jump distribution satisfies
\[
\mu(z^{-1}xz)=\mu(x)\qquad \text{for all } x,z\in G.
\]

Conjugacy-invariant random walks have been studied in a variety of finite-group settings. This property is automatic in the Abelian case, making Abelian groups the simplest examples. A classical non-Abelian example is the random transposition walk on the symmetric group $G = \mathcal{S}_n$, in which each jump consists of selecting two distinct labels uniformly at random and swapping them, equivalently right-multiplying by a uniformly chosen transposition, so that the jump distribution is supported on the conjugacy
class of transpositions. The first proof of cutoff for this chain was given by Diaconis and Shahshahani \cite{diaconis1981generating} using representation theory.
Subsequently, a probabilistic perspective was developed in the work of Berestycki, Schramm, and Zeitouni \cite{BerestyckiSchrammZeitouni2011} and extended by Berestycki and {\c{S}}eng{\"u}l \cite{berestycki2019cutoff}, providing a general understanding of conjugacy-invariant walks on $ \mathcal{S}_n$.

For conjugacy-invariant walks on \emph{nilpotent groups}, early progress likewise relied on representation-theoretic tools. A natural and extensively studied example is the group $\mathbb{U}_n(p)$ of $n$-dimensional  unit upper-triangular matrices over $\mathbb{Z}_p:=\ZZ/p\ZZ$, whose character theory is notoriously intricate and long posed a significant barrier to the analysis of such walks. 
The introduction of superclass theory by Andr\'e, Carter, and Yan
(see, e.g., \cite{andre1995basic, andre1998regular, andre2002basic,
yan2001representation, yan2010representations}) provided a tractable
representation-theoretic framework for algebra groups based on
superclasses and supercharacters, which form a coarser analogue of
conjugacy classes and irreducible characters. This framework is well suited to conjugacy-invariant random walks and
was applied by Arias-Castro, Diaconis, and Stanley \cite{arias2004super}
to analyze a conjugacy-invariant random walk with jump distribution
\beq\label{superclass_walk}
\mu(x)=
\begin{cases}
\displaystyle \frac{1}{2(n-1)p^{\,n-2}}, & \text{if } x\in \cup_{i=1}^{n-1}C_i(\pm 1),\\
0,& \text{otherwise},
\end{cases}
\eeq
where $E_{i,i+1}$ denotes the $n\times n$ matrix with a $1$ in the $(i,i+1)$ entry
and zeros elsewhere, and $
C_i(\pm 1)
= \{ x (I \pm E_{i,i+1}) x^{-1} : x\in \mathbb{U}_n(p) \}
$
denotes the conjugacy class of $I\pm E_{i,i+1}$. This walk is shown to have mixing time of  order
$\Theta(p^2 n\log n)$.  Later,  Nestoridi \cite{nestoridi2019super} analyzed a similar walk with jump distribution 
\beq\label{superclass_walk_nestoridi}
\nu(x)=
\begin{cases}
\displaystyle \frac{1}{4(n-1)p^{\,n-2}}, & \text{if } x\in \cup_{i=1}^{n-1}(C_i(\pm 1)\cup C_i(\pm \sqrt{a})),\\
0,& \text{otherwise},
\end{cases}
\eeq
where $a=\lfloor \sqrt{p} \rfloor+\1\{\lfloor \sqrt{p} \rfloor \text{ is even}\}$ and $C_i(\pm a)=\{ x (I \pm a E_{i,i+1}) x^{-1} : x\in \mathbb{U}_n(p) \}$, and proved an upper bound $O(p n \log n)$ on the mixing time.

Despite the  involved representation-theoretic machinery used in \cite{arias2004super, nestoridi2019super} to bound mixing times, the mechanism governing the convergence of both walks appears to be considerably simpler: the walks behave similarly to their respective projections
\[
X_t^{\ab} :=[G,G] X_t \quad \text{with } G = \mathbb{U}_n(p),
\]
onto the abelianization $G_{\ab} := G/[G,G]$. (Since $[G,G]$ is normal in $G$, we
freely identify left and right cosets throughout.) This quotient is naturally
isomorphic to the first superdiagonal of the matrix, and hence to
$\mathbb{Z}_p^{n-1}$, so that the projected walks are equivalent to product chains on $\ZZ_p^{n-1}$ where each coordinate evolves independently. Using this observation, it is straightforward to recover the mixing-time bounds established in \cite{arias2004super, nestoridi2019super}. Moreover, we will see in Corollary \ref{cutoff} that the same mechanism gives rise to cutoff for both walks.

More generally, one asks whether mixing and cutoff for any conjugacy-invariant random walk on a finite nilpotent group $G$ can be understood through its projected process on the abelianization $G_{\ab}$. Theorem~\ref{reduction_abelianization} addresses this question by establishing a quantitative comparison between the mixing times of the walk on $G$ and those of its projection onto $G_{\ab}$.

 \bigskip
 Let $P$ be the transition matrix of the random walk on a finite nilpotent group $G$, and define the associated
continuous-time transition kernel by
$
P_t(x,y):=\sum_{k=0}^\infty \frac{t^k}{k!} P^k(x,y)\, e^{-t} 
$
for $x,y\in G$ and $t\in \mathbb{R}_+$. By (right) invariance of the walk, we may assume without loss of generality that it starts
at the identity $\id$. Let $\pi$ denote the uniform measure on $G$, which is
the stationary distribution. Throughout, we assume that the walk is irreducible.

The total variation mixing time is defined by
\beq\label{tv_mixing_time}
t^{\TV}_{\mathrm{mix}}(G,\varepsilon)
:=\inf\left\{t\geq 0:\, \| P_t(\id,\cdot)-\pi\|_{\TV}\le\varepsilon\right\},
\qquad \varepsilon\in(0,1),
\eeq
where $
\| P_t(\id,\cdot)-\pi\|_{\TV}:=\frac{1}{2}\sum_{x\in G}\bigl|P_t(\id,x)-\pi(x)\bigr|
$ denotes the total variation distance.
We also define the $\ell^2$-mixing time by
\beq\label{ell2_mixing_time}
t_{\mathrm{mix}}^{\ell^2}(G,\varepsilon)
:=\inf\left\{t\geq 0:\,
\left\| \frac{P_t(\id,\cdot)}{\pi(\cdot)}-1\right\|_{2,\pi}\le\varepsilon
\right\},
\qquad \varepsilon>0,
\eeq
where $\|f\|_{2,\pi}
:=\left(\sum_{x\in G} \pi(x)\, f^2(x)\right)^{1/2}$. For the projected walk $X^\ab=[G,G]X$ on $G_{\ab}$, the corresponding mixing times are denoted by $t^{\TV}_{\mathrm{mix}}(G_{\ab},\varepsilon)$ and
$t^{\ell^2}_{\mathrm{mix}}(G_{\ab},\varepsilon)$.

\smallskip
For any conjugacy-invariant random walk $X=(X_t)_{t\ge 0}$ with jump
distribution $\mu$, the support of $\mu$ is a union of disjoint conjugacy
classes of $G$. Consequently, there exists a set
$S=\{s_a : a\in[k]\}\subseteq G$ containing exactly one representative from
each conjugacy class intersecting $\operatorname{supp}(\mu)$ such that
\beq\label{def_S}
\operatorname{supp}(\mu)
= \bigsqcup_{a\in [k]} \Cl(s_a),
\eeq
where $\Cl(s)=\{x^{-1} s x : x\in G\}$ denotes the conjugacy class of $s$.
The set $S$ is unique up to choice of representatives for these conjugacy
classes. The assumption of irreducibility of the walk is equivalent to the condition that $S$ generates 
$G$.

We obtain the following comparison result for conjugacy-invariant random walks on a nilpotent group $G$; notably, no reversibility assumption on $\mu$ is required. 

\begin{theorem}\label{reduction_abelianization}
Consider a rate-$1$ conjugacy-invariant random walk on a finite nilpotent group $G$ with jump distribution $\mu$ and an associated set $S=\{s_a : a \in [k]\}$ as in \eqref{def_S}, with uniform stationary distribution. Let $\mu_*:=\min_{a\in [k]} \mu(\Cl(s_a))$. Then, for every $\varepsilon\in(0,1)$, we have
\beq\label{mixing_time_comparison}
t_{\mix}^{\TV}(G_{\ab},\varepsilon)
\leq
t_{\mix}^{\TV}(G,\varepsilon)
\leq
 \max\{ t_{\mix}^{\ell^2}(G_{\ab},\varepsilon/2), \mu_*^{-1}(\log k+2\log(4/\ep))\}.
\eeq
\end{theorem}

\medskip
 In cases where the total variation and $\ell^2$ mixing times of the projected
walk $X^\ab$ coincide up to lower-order errors, Theorem~\ref{reduction_abelianization} immediately
implies the occurrence of cutoff. More precisely, a sequence of Markov chains $(X^{(n)})_{n\ge 1}$ on $\Omega^{(n)}$ is said to exhibit
\emph{cutoff} at times $t_n$ with window size $w_n=o(t_n)$ if
\[
\lim_{c\to\infty}\,\liminf_{n\to\infty}
d^{(n)}_{\mathrm{TV}}(t_n-cw_n)=1
\quad\text{and}\quad
\lim_{c\to\infty}\,\limsup_{n\to\infty}
d^{(n)}_{\mathrm{TV}}(t_n+cw_n)=0,
\]
where
$
d^{(n)}_{\mathrm{TV}}(t)
:=\max_{x\in \Omega^{(n)}}
\| \P_x( X_t^{(n)}=\cdot)-\pi^{(n)}(\cdot)\|_{\TV}
$
denotes the total variation distance to the stationary distribution
$\pi^{(n)}$. 

To simplify notation, we drop the superscript $(n)$ in what follows, as the meaning will be clear from context.

As a consequence of Theorem~\ref{reduction_abelianization}, we obtain cutoff for the walks \eqref{superclass_walk} and \eqref{superclass_walk_nestoridi}.

\begin{corollary}\label{cutoff}
Let $(\mathbb{U}_n(p))_{n\ge 1}$ be a sequence of groups of unit upper-triangular
$n\times n$ matrices over $\mathbb{Z}_p$, where $p=p_n$ is an integer that may
depend on $n$ and satisfies $p\ge 6$.
Consider the associated sequence of continuous-time, rate-$1$ random walks.
\begin{enumerate}[label=(\alph*)]
\item The walk driven by the jump distribution $\mu$ in \eqref{superclass_walk}
exhibits cutoff at time
\[
\frac{(n-1)\log (n-1)}{2\bigl(1-\cos(2\pi/p)\bigr)}.
\]

\item The walk driven by the jump distribution $\nu$ in \eqref{superclass_walk_nestoridi}
exhibits cutoff at time
\[
\frac{(n-1)\log (n-1)}{(1-\cos(2\pi/p))+(1-\cos(2\pi\lfloor \sqrt{a}\rfloor/p))},
\]
where $a=\lfloor \sqrt{p} \rfloor+\1\{\lfloor \sqrt{p} \rfloor \text{ is even}\}$.
\end{enumerate}
\end{corollary}

\section{Related work}

\subsection{Abelianization as the dominant mixing mechanism}

The mixing behavior of random walks on groups of unit upper-triangular matrices has been extensively studied, as these groups form a canonical class of nilpotent groups. A recurring theme in this line of work is that the mixing behavior of random walks on $\mathbb{U}_n(p)$ (where $p$ is an integer) is often governed by their projection onto the abelianization, corresponding to the superdiagonal entries. Throughout this section, unless stated otherwise, we consider simple random walks whose jump distribution is uniform on its support, which we also refer to as the generating set of the walk.

One concrete instance is provided by Nestoridi and Sly \cite{nestoridi2023random}, who consider the simple random walk on $\mathbb{U}_n(p)$ generated by elementary row operations
$S = \{I \pm E_{i,i+1}:i \leq n-1\}$, where $E_{i,i+1}$ denotes the matrix with a single $1$ in the
$(i,i+1)$ entry and zeros elsewhere. Their analysis upper bounds the mixing time to be 
$O(p^2 n \log n + n^2 p^{o(1)})$, with the first term arising from the abelianization: the projected walk onto the abelianization is a simple random walk on $\ZZ_p^{n-1}$, which mixes in time $O(p^2n\log n)$.
This perspective also appears in the work of Hermon and Olesker-Taylor \cite{hermon2019cutoff}, who consider random Cayley graphs of $\mathbb{U}_n(p)$ with generating sets formed by i.i.d.\ uniformly distributed group elements. In this setting, and under mild regularity assumptions, they establish cutoff with high probability over the choice of the generating set and obtain sharp asymptotics for the mixing time, once again matching those of the projected walk onto the abelianization. With a different emphasis, Diaconis and Hough \cite{diaconis2021random} develop a general framework for random walks on $\mathbb{U}_n(p)$ with symmetric jump distributions under certain technical assumptions. Their results yield a precise, coordinate-wise description of mixing, showing in particular that coordinates on the $k$-th diagonal require order $p^{2/k}$ steps to mix. In this sense, mixing on the abelianization (i.e., $k=1$) plays a more prominent role than that of higher diagonals.

More broadly, a series of influential works \cite{diaconis1994moderate, diaconis2020application, diaconis1996nash} by Diaconis and Saloff-Coste develop methods for studying mixing on \emph{nilpotent groups} through functional inequalities and geometric ideas. In particular, for generating sets of bounded size, they establish that diameter-squared steps are both necessary and sufficient to reach uniformity, and that cutoff does not occur in this case. However, these bounds become significantly looser as the size of the generating set grows, and therefore do not directly apply to walks on matrix groups in the regime where the dimension $n$ diverges. 

In the author's prior work with Hermon \cite{hermon2024cutoff}, the regime of growing generating sets is studied for general nilpotent groups. Let $G$ be a finite nilpotent group, $r=r(G)$ its rank, and $L=L(G)$ its nilpotency class as defined in \eqref{lower_central_series}. It is shown that if the generating set has size at most $\frac{\log |G|}{8L r^L \log\log|G|}$, then the mixing time is determined, up to lower-order terms, by the dynamics of the projection onto the abelianization. For conjugacy-invariant walks on nilpotent groups, such as \eqref{superclass_walk}, the generating set is typically much larger and therefore falls outside this regime; nevertheless, we will see in this paper that the abelianization continues to play a dominant role.
Indeed, this dominance of the abelianization appears to be typical for random walks on nilpotent groups. When the generating set consists of i.i.d.\ uniformly distributed elements and under mild structural assumptions of the nilpotent groups, \cite{hermon2024cutoff} shows that the walk exhibits cutoff (with high probability) with sharp mixing time asymptotics determined by the dynamics of the projection onto the abelianization.

\subsection{Entropic criteria for cutoff}
A remarkable sequence of recent breakthroughs by Salez and coauthors \cite{salez2023cutoff, salez2025cutoff, pedrotti2025new} develops a conceptually profound understanding of the cutoff phenomenon, establishing sufficient conditions for cutoff in non-negatively curved Markov chains without requiring precise identification of the mixing time.

For a sequence $(X^{(n)})_{n\ge 1}$ of Markov chains with nonnegative
\emph{carr\'e du champ curvature} (see \cite{pedrotti2025new} for a detailed exposition), and whose transition kernels $P^{(n)}$ are symmetric
(i.e., $P^{(n)}(x,y)>0$ if and only if $P^{(n)}(y,x)>0$), the following
sufficient criterion for cutoff is established in \cite{pedrotti2025new}:
\beq\label{cutoff_criterion}
\frac{t^{(n)}_{\mix}}{ t^{(n)}_{\mathrm{MLS}} \log\log d^{(n)}}\to \infty \quad \text{ as $n\to\infty$}.
\eeq
Here $t_{\mix}=t^{(n)}_{\mix}:=t^{(n)}_{\mix}(1/4)$ denotes the total variation mixing time, and $d=d^{(n)}:=\max_{x\sim y}\frac{1}{P^{(n)}(x,y)}$ denotes the sparsity parameter of the transition kernel $P^{(n)}$, and $t_{\mathrm{MLS}}=t^{(n)}_{\mathrm{MLS}}$ denotes the inverse modified log-Sobolev constant. As discussed in Section~2 of \cite{pedrotti2025new}, this criterion applies in a wide range of settings and can be verified for many classical examples, including several conjugacy-invariant random walks on groups. Such walks are known to have nonnegative carr\'e du champ curvature; see Theorem 2 of \cite{hermon2024concentration}.
The examples presented in Section~2.1 of \cite{pedrotti2025new} further suggest that the condition \eqref{cutoff_criterion} is quite sharp, in the sense that the orders of $t_{\mix}$ and $t_{\mathrm{MLS}} \log\log d$ typically differ by at most logarithmic factors. 

Interestingly, our results provide an example, namely the walk \eqref{superclass_walk} on $\mathbb{U}_n(p)$, in which cutoff occurs even though condition \eqref{cutoff_criterion} fails. Suppose $p=p_n$ diverges as $n\to\infty$. It is a well-known fact (see, e.g., (3.15) of \cite{salez2025modernaspectsmarkovchains}) that for reversible Markov chains
\[
t_{\mathrm{MLS}} \ge \frac{ t_{\mathrm{rel}}}{2} \ge \frac{ t_{\mathrm{rel}}^{\ab}}{2} \asymp p^2 n,
\]
where $t_{\mathrm{rel}}$ denotes the relaxation time of the walk and $t_{\mathrm{rel}}^{\ab}$ that of its projection onto the abelianization. Moreover, the sparsity parameter of this walk is
$d = 2(n-1)p^{n-2}$, which implies
$\log\log d \asymp \max\{\log n, \log\log p\}$. On the other hand, by Corollary~\ref{cutoff}, we have $t_{\mix} \asymp p^2 n \log n$. Consequently, condition~\eqref{cutoff_criterion} does not hold, despite the presence of cutoff.

By contrast, the projected walk onto the abelianization is a product walk on $\ZZ_p^{\,n-1}$ with sparsity parameter $2(n-1)$, for which the criterion \eqref{cutoff_criterion} is satisfied and cutoff occurs. This contrast can be traced to the substantially larger sparsity parameter of the original chain, namely $2(n-1)p^{n-2}$, which is not sharp for determining cutoff in this setting. The key takeaway from this example is that, in the setting of conjugacy-invariant walks on nilpotent groups, the sufficient condition \eqref{cutoff_criterion} may admit a refinement that depends only on the abelianization.

\section{Preliminaries}

We begin by introducing the notation that will be used throughout.
 The lower central series of a group $G$ is defined
recursively by
\[
G_1 := G, \qquad \text{and}\qquad
G_{k+1} := [G_k, G] = \langle [x,y] : x\in G_k,\ y\in G \rangle  \quad \text{for $k\ge 1$},
\]
where $[x,y]=x^{-1}y^{-1}xy$. A group $G$ is called \emph{nilpotent} if its lower central series terminates after finitely many steps, namely
\beq\label{lower_central_series}
G = G_1 \trianglerighteq G_2 \trianglerighteq \cdots \trianglerighteq
G_L \trianglerighteq G_{L+1} = \{\id\}
\eeq
for some $L=L(G)$ that is referred to as the \emph{nilpotency class} of $G$. In this case, each quotient in
$\{G_\ell/G_{\ell+1}\}_{\ell\in[L]}$ is Abelian. In particular, let $G_{\ab}:=G/[G,G]$ denote the \emph{abelianization} of $G$. We will assume throughout that $G$ is a finite nilpotent group.

The following proposition collects several standard properties of commutators and nilpotent groups. 
Proofs may be found in many references, see e.g. \cite{dummit2004abstract}, and are omitted here. Part~(i) of Proposition~\ref{nilpotent_properties} is a fairly well known result obtained by an induction argument based on the three subgroup lemma, while part~(iii) follows from straightforward computations using the identity $[x,zy]=[x,z][x,y][z,[y,x]]^{-1}$ for $x,y,z\in G$.

 \begin{proposition}\label{nilpotent_properties}
Let $G$ be a finite nilpotent group. Then:
\begin{enumerate}[label=(\roman*),leftmargin=*, itemindent=0pt]
\item The lower central series of a nilpotent group $G$ is a strongly central series, i.e., $[G_i,G_j]$ is a subgroup of $G_{i+j}$  for all $i,j\geq 1$.
\item A set $S\subseteq G$ generates $G$ if and only if its projection $G_2 S:=\{ G_2s: s\in S\}$ generates the abelianization $G_{\mathrm{ab}}$. 
\item For $\ell\geq 2$, the map $\phi : G \times G_{\ell-1}  \to G_{\ell}/G_{\ell+1}$
given by $\phi(g,h) := G_{\ell+1}[g,h]$ is anti-symmetric and bi-linear. Namely, the following hold
for all $x,y\in G$ and $z,w\in G_{\ell-1}$:
\begin{align*}
G_{\ell+1}[x,z]&=G_{\ell+1}[z,x]^{-1},\\
G_{\ell+1}[x,zw] &= G_{\ell+1}[x,z]\;G_{\ell+1}[x,w],\\
G_{\ell+1}[xy,z]&=G_{\ell+1}[x,z]\;G_{\ell+1}[y,z],\\
G_{\ell+1}[x^i,z^j] &= G_{\ell+1}[x,z]^{i j}
\quad\text{for all } i,j \in \mathbb Z.
\end{align*}

\end{enumerate}
\end{proposition}

Let $N(t)$ be a Poisson random variable with mean $t$, and let
$S=\{s_a : a\in[k]\}$ be the representative set from \eqref{def_S} for the jump
distribution $\mu$ of $X=(X_t)_{t\ge 0}$. For a collection $\{U_i\}_{i\ge 1}$ of i.i.d.\ uniform
elements over $G$, it is straightforward to see that
$
U_i^{-1} s\, U_i \overset{\mathrm{i.i.d.}}{\sim} \Unif(\Cl(s)).
$
Thus, the jumps of $X$ can be written as $\{U_i^{-1} s_{\sigma_i} U_i\}_{i\ge 1}$, where
$\{\sigma_i\}_{i\ge 1}$ are i.i.d.\ with
\[
\P(\sigma_i = a) = \mu\bigl(\Cl(s_a)\bigr), \qquad a\in[k],
\]
where $\mu$ need not be uniform across conjugacy classes.
Consequently, $X_t$ admits the representation
\beq\label{seq_X}
X_t
= \prod_{i=1}^{N(t)} \bigl(U_i^{-1} s_{\sigma_i} U_i\bigr).
\eeq

Inspired by the approach of \cite{hermon2024cutoff}, we analyze $X_t$ via the decomposition arising from the quotient series $\{G_\ell/G_{\ell+1}\}_{\ell\in[L]}$.  To this end, it is convenient to express the i.i.d.\ uniform variables $\{U_i\}_{i\ge1}$ in terms of the same quotient structure.   For each $\ell\in[L]$, let $R_\ell \subseteq G_\ell$ be a set of \emph{representatives} for the quotient $G_\ell/G_{\ell+1}$; that is, $|R_\ell| = |G_\ell/G_{\ell+1}|$ and $G_\ell/G_{\ell+1} = \{ G_{\ell+1} g : g \in R_\ell \}.$

\begin{lemma}[Corollary~6.4 in \cite{hermon2102cutoff}]\label{U_decomp}
Let $\{U^{(\ell)}\}_{\ell\in[L]} $ be independent random variables with $U^{(\ell)} \sim \Unif(R_\ell)$.
Then we have $G_{\ell+1} U^{(\ell)} \sim \Unif(G_\ell/G_{\ell+1})$. Moreover, the product 
$
U := \prod_{\ell=1}^L U^{(\ell)}
$
is uniformly distributed on $G$.
\end{lemma}

In addition, we collect a few basic group-theoretic facts that will
be used later. 
\begin{lemma}\label{group_basics}
\begin{enumerate}[label=(\roman*)]
\item For fixed $s \in G$ and $U \sim \Unif(R_{\ell-1})$ with $\ell \ge 2$, $G_{\ell+1}[s,U]$ is uniformly distributed on the subgroup $\{G_{\ell+1}[s,u] : u \in R_{\ell-1}\}$ of $G_\ell/G_{\ell+1}$.
\item 
Let $Q$ be a finite abelian group (written additively), and let $H_1,\dots,H_r$ be subgroups of $Q$.  
If $\{V_i\}_{1\le i\le r}$ are independent with $V_i \sim \Unif(H_i)$ for each $i$, then
\[
\sum_{i=1}^r V_i \sim \Unif(H) \qquad \text{ where $H := \{h_1+\cdots+h_r : h_i \in H_i\}$.}
\]

\end{enumerate}
\end{lemma}
\begin{proof}

For fixed $s\in G$, note that the map 
$\phi_s : G_{\ell-1}/G_{\ell} \to G_{\ell}/G_{\ell+1}$ given by
$$
\phi_s(G_\ell x) := G_{\ell+1}[s,x], \qquad x \in G_{\ell-1},
$$
is well defined and a group homomorphism with image  $\mathrm{Im}(\phi_s)=\{G_{\ell+1}[s,u] : u \in G_{\ell-1}\}=\{G_{\ell+1}[s,u] : u \in R_{\ell-1}\}.$
Note from Lemma \ref{U_decomp} that $G_\ell U\sim \Unif(G_{\ell-1}/G_{\ell})$ since $U\sim \Unif(R_{\ell-1})$.  
For any element $G_{\ell+1}[s,y] \in \{G_{\ell+1}[s,u] : u \in R_{\ell-1}\}$, we compute
\begin{align*}
\P(\phi_s(G_\ell U)=G_{\ell+1}[s,y])&=\P(G_{\ell+1}[s,Uy^{-1}]=G_{\ell+1})=\P(\phi_s(G_\ell (Uy^{-1}))=G_{\ell+1})\\
&=\P(\phi_s(G_\ell U)=G_{\ell+1})=\frac{|\mathrm{Ker}(\phi_s)|}{|G_{\ell-1}/G_{\ell}|},
\end{align*}
which  does not depend on $y$.   Hence $\phi_s(G_\ell U)=G_{\ell+1}[s,U]$ is uniform on
$\{G_{\ell+1}[s,u]:u\in R_{\ell-1}\}$.

The proof of (ii) follows from an analogous argument upon noting that $\varphi: H_1 \times \cdots \times H_r \to Q$ with 
$\varphi(h_1,\dots,h_r) := h_1 + \cdots + h_r$ is a group homomorphism. Thus for any $x\in H$,
$$\P(\varphi(V_1,\dots,V_r)=x)=\frac{|\varphi^{-1}(x)|}{|H_1|\cdots |H_r|}=\frac{|\mathrm{ker}(\varphi)|}{|H_1|\cdots |H_r|}.$$
 As $\varphi(V_1,\dots,V_r)$ takes values in $H$ with equal probability, it is uniformly distributed over $H$.
\end{proof}

\subsection{Quotient decomposition of the walk}

By a standard application of the Cauchy--Schwarz inequality, one has
\[
4\,\|\P(X_t=\cdot)-\pi\|_{\mathrm{TV}}^{2}
    \le |G|\,\P(X_t = X'_t) - 1,
\]
where $X'=(X'_t)_{t\geq 0}$ is an independent copy of the walk $X=(X_t)_{t\geq 0}$.
This naturally leads us to study the product $X_t (X'_t)^{-1}$. 

Let $N'=N'(t)$
denote the number of jumps of $X'$ up to time $t$, and define the sequences
$\{\sigma'_i\}_{i\ge1}$ and $\{U'_i\}_{i\ge1}$ for $X'$ analogously to
$\{\sigma_i\}_{i\ge1}$ and $\{U_i\}_{i\ge1}$ in \eqref{seq_X}. For ease of notation,
we suppress the dependence on $t$ when it is clear from context. We may then
write
\beq\label{seq_X_rewritten}
X(X')^{-1}=\prod_{i=1}^{N}  \bigl(U_i^{-1} s_{\sigma_i} U_i\bigr)  \left( \prod_{j=1}^{N'}  \bigl((U'_j)^{-1} s_{\sigma'_j} U'_j\bigr) \right)^{-1}=:\prod_{i=1}^{N+N'}  \bigl(U_i^{-1} s^{\eta_i}_{\sigma_i} U_i\bigr), 
\eeq
where in the second equality we simply unify notation via relabeling and define $\eta_i =1$ for $1\leq i\leq N$ and $\eta_i=-1$ otherwise.

Following the decomposition in Lemma~\ref{U_decomp}, we construct the uniform elements
$\{U_i\}_{i\ge1}$ as follows. For each $\ell\in[L]$ and $i\ge1$, sample
$
\{U_{i,\ell}\}_{i\ge1} \overset{\mathrm{i.i.d.}}{\sim} \Unif(R_\ell),
$
and define
\[
U_i := \prod_{\ell=1}^L U_{i,\ell}, \qquad \text{ for }i\ge1.
\]
Let $\{U'_i\}_{i\ge1}$ be constructed analogously from an independent collection
$\{U'_{i,\ell}\}_{i\ge1,\;\ell\in[L]}$.

We now introduce several $\sigma$-fields that will be used in the subsequent analysis. Let $t\geq 0$ be fixed and we suppress the dependence on $t$ in the following definitions.
\begin{definition}\label{filtration}
\begin{enumerate}[label=(\roman*),leftmargin=*, itemindent=0pt]
\item Let $\widetilde{\HH}$ be the $\sigma$-field generated by $N, N'$, and the sequence $\{\sigma_i\}_{i\in[N]}, \{\sigma'_i\}_{i\in [N']}$; in other words, $\widetilde{\HH}$ contains all information about the sequences $X$ and $X'$, apart from the identities of the variables $\{U_i\}_{i\in [N]}, \{U'_i\}_{i\in [N']}$.

\item For each $\ell\in[L]$, let
\[
\FF_\ell := \sigma\bigl(\{U_{i,j}, U'_{i,j} : i\ge1,\; 1\le j\le \ell\}\bigr),
\]
and let $\FF_0$ denote the trivial $\sigma$-field.
\end{enumerate}
\end{definition}

As a final preparation for the proof, we clarify the measurability of the events of interest. For each $\ell \in [L+1]$, define
\beq\label{event_E}
\EE_\ell := \{ X (X')^{-1} \in G_\ell \}
\eeq
and note, in particular, that $\EE_{L+1}=\{X=X'\}$.
Let $\widetilde{\HH}$ and $\{\FF_i\}_{0 \le i \le L}$ be the $\sigma$-fields introduced in Definition~\ref{filtration}.  

\begin{lemma}\label{E_measurability}
For $2 \le \ell \le L+1$, define
\begin{equation}\label{seq_phi}
\varphi^{(\ell-2)}
  := \prod_{i=1}^{N+N'} s_{\sigma_i}^{\eta_i}
      \Big[ s_{\sigma_i}^{\eta_i},\; \prod_{j \le \ell-2} U_{i,j} \Big].
\end{equation}
Then we have $\EE_\ell = \{ \varphi^{(\ell-2)} \in G_\ell \}$,
and consequently, for each $2 \le \ell \le L+1$, the event $\EE_\ell$ is measurable with respect to
$\sigma(\FF_{\ell-2}, \widetilde{\HH})$.
\end{lemma}

\begin{proof}
Since $G_\ell \trianglelefteq G$, we can rewrite \eqref{seq_X_rewritten} as
\beq\label{seq1}
G_\ell X (X')^{-1}
  = \prod_{i=1}^{N+N'} G_\ell \big( U_i^{-1} s_{\sigma_i}^{\eta_i} U_i \big)
  = \prod_{i=1}^{N+N'} G_\ell \big( s_{\sigma_i}^{\eta_i} [s_{\sigma_i}^{\eta_i}, U_i] \big),
\eeq
where the second equality uses the elementary identity $u^{-1} s u = s[s,u]$ for any $u,s\in G$. 

For any $x\in G$ and any $j\ge \ell-1$, since $U_{i,j}\in G_j$ we have
$
[x,U_{i,j}] \in G_{j+1} \subseteq G_\ell.
$
Write $U_i=\prod_{j=1}^L U_{i,j}$ as in Lemma~\ref{U_decomp}, and decompose
\[
U_i = A_i B_i,
\qquad\text{where}\qquad
A_i:=\prod_{j\le \ell-2} U_{i,j}
\quad\text{and}\quad
B_i:=\prod_{j\ge \ell-1} U_{i,j}.
\]
Since $G_\ell$ is normal, we have $[x,B_i] \in G_\ell$. Applying the commutator identity $[x,ab] = [x,b]\,(b^{-1}[x,a] b)$ for all $x,a,b \in G$, we obtain
$$
G_\ell[s_{\sigma_i}^{\eta_i},U_i]
=
G_\ell[s_{\sigma_i}^{\eta_i},A_iB_i]
=
G_\ell\bigl([s_{\sigma_i}^{\eta_i},B_i] (B_i^{-1}[s_{\sigma_i}^{\eta_i},A_i]B_i)\bigr)
=
G_\ell[s_{\sigma_i}^{\eta_i},A_i],
$$
where the last equality again follows from $G_\ell \trianglelefteq G$.
Substituting this into \eqref{seq1} yields
\begin{align}\label{quotient_simplification}
G_\ell X (X')^{-1}
&= \prod_{i=1}^{N+N'} G_\ell\Big( s_{\sigma_i}^{\eta_i}\,[s_{\sigma_i}^{\eta_i},A_i] \Big) \nonumber\\
&= G_\ell\left(  \prod_{i=1}^{N+N'} s_{\sigma_i}^{\eta_i}\,
      \Big[ s_{\sigma_i}^{\eta_i},\; \prod_{j\le \ell-2} U_{i,j} \Big] \right)
= G_\ell \varphi^{(\ell-2)}.
\end{align}
This establishes $
\EE_\ell
=\{\varphi^{(\ell-2)}\in G_\ell\}.$ Finally, we observe that $\varphi^{(\ell-2)}$ is measurable with respect to 
$\sigma(\FF_{\ell-2}, \widetilde{\HH})$, and therefore $\EE_\ell$ is measurable
with respect to the same $\sigma$-field.

\end{proof}

\section{Proof of main results}
\subsection{Proof of Theorem \ref{reduction_abelianization}}
\begin{proof}
The first inequality $t_{\mix}^{\TV}(G_{\ab},\varepsilon)
\leq
t_{\mix}^{\TV}(G,\varepsilon)$ follows directly from the fact that total variation distance is non-increasing under projection. We proceed to showing the second inequality. 

Let  $X=(X_t)_{t\geq 0}$ denote the rate-1 random walk and let $S=\{s_a: a\in [k]\}\subseteq G$ denote the  set associated with the jump distribution $\mu$ as defined in \eqref{def_S}; the uniformity of the stationary measure implies that $S$ generates $G$.
For each $a\in [k]$, let $N_a=N_a(t)$ denote the number of times elements in the conjugacy class $\Cl(s_a)=\{ x^{-1}s_a x:x\in G\}$ appears in the walk $X$ up to time $t$, which follows a Poisson distribution with rate $\mu\bigl(\Cl(s_a)\bigr)$. 

Define the event $\AA=\AA(t):=\{ \{G_2s_a: a\in [k], N_a(t)>0\} \text{ generates } G\}$. By the triangle inequality,
\beq\label{tv_triangle_inequality}
\left\| \P(X_t=\cdot)-\pi \right\|_{\mathrm{TV}}
\leq \left\| \P(X_t=\cdot| \AA)-\pi \right\|_{\mathrm{TV}}
+ \P(\AA^c).
\eeq
Let $X'=(X'_t)_{t\geq 0}$ be an independent copy of $X$, and let
$(N'_a)_{a\in [k]}$ and the event $\AA'$ be defined for $X'$ analogously.
By the Cauchy--Schwarz inequality, we obtain
\beq\label{L2_bound}
4 \left\|\P(X_t=\cdot | \AA)-\pi \right\|^2_{\mathrm{TV}}
\leq |G|\cdot \P\!\left(X_t=X'_t | \AA\cap\AA'\right)-1.
\eeq
For notational simplicity, we suppress the time index $t$ from now on.  Recall from \eqref{event_E} that $\EE_\ell=\{X(X')^{-1}\in G_\ell\}$ for $\ell\in[L+1]$.
It thus remains to upper bound
\beq\label{D(t)}
D(t)
:= |G|\cdot \P\!\left(X(X')^{-1}=\id | \AA\cap\AA'\right)-1
= |G|\cdot \P(\EE_{L+1}| \AA\cap\AA')-1 .
\eeq

We start by simplifying $G_{\ell+1}X(X')^{-1}$ for each $\ell\in[L]$, following the approach of
Lemma~\ref{E_measurability}. Recall that each uniform element $U_i$ in the collection
$\{U_i\}_{i\geq 1}$ is defined via the decomposition $U_i=\prod_{\ell=1}^L U_{i,\ell}$ where $U_{i,\ell}\overset{\mathrm{i.i.d.}}{\sim} \Unif(R_\ell).$ A similar argument to that in
\eqref{quotient_simplification} yields
\begin{align*}
G_{\ell+1}X(X')^{-1}&=\prod_{i=1}^{N+N'} G_{\ell+1} \left(s^{\eta_i}_{\sigma_i} [s_{\sigma_i}^{\eta_i}, \prod_{j\leq \ell-1} U_{i,j}]\right).
\end{align*}
As $ U_{i,\ell-1}\in R_{\ell-1}\subseteq G_{\ell-1}$, we have $[s_{\sigma_i}^{\eta_i}, U_{i,\ell-1}]\in G_{\ell}$. It is not difficult to see that $G_{\ell+1}[s_{\sigma_i}^{\eta_i}, U_{i,\ell-1}]$ commutes with any element in $G/G_{\ell+1}$ (as their commutator lies in $G_{\ell+1}$). Hence, using $G_{\ell+1} \trianglelefteq G$, the above can be further simplified to 
\begin{align}\label{simplified_quotient}
\nonumber G_{\ell+1}X(X')^{-1}&=\prod_{i=1}^{N+N'} G_{\ell+1} \left(s^{\eta_i}_{\sigma_i} [s_{\sigma_i}^{\eta_i}, \prod_{j\leq \ell-2} U_{i,j}][s_{\sigma_i}^{\eta_i},U_{i,\ell-1}]\right)\\
&= G_{\ell+1} \varphi^{(\ell-2)} \cdot G_{\ell+1} \left( \prod_{i=1}^{N+N'} [s_{\sigma_i}^{\eta_i},U_{i,\ell-1}]\right)=:G_{\ell+1} \varphi^{(\ell-2)} \cdot G_{\ell+1} f^{(\ell-1)},
\end{align}
where $\varphi^{(\ell-2)}= \prod_{i=1}^{N+N'} s_{\sigma_i}^{\eta_i}
      [ s_{\sigma_i}^{\eta_i},\; \prod_{j \le \ell-2} U_{i,j}]$ as in  \eqref{seq_phi}, and we define $f^{(\ell-1)}:=\prod_{i=1}^{N+N'} [s_{\sigma_i}^{\eta_i}, U_{i,\ell-1}]\in G_\ell.$

Note that both events $\AA$ and $\AA'$ are measurable with respect to the
$\sigma$-field $\widetilde\HH$ in Definition~\ref{filtration}. By
Lemma~\ref{E_measurability}, the event
$\EE_\ell=\{\varphi^{(\ell-2)}\in G_\ell\}$ is measurable with respect to
$\sigma(\widetilde\HH,\FF_{\ell-2})$.  We can therefore observe from
\eqref{simplified_quotient} that
\begin{align}\label{conditional_prob}
\nonumber\1\{\EE_{\ell}\cap\AA\cap\AA'\} \cdot \P(\EE_{\ell+1}|\FF_{\ell-2},\widetilde\HH)&=\1\{\EE_{\ell}\cap\AA\cap\AA'\} \cdot\P(G_{\ell+1}\varphi^{(\ell-2)}\cdot G_{\ell+1}f^{(\ell-1)}=G_{\ell+1}|\FF_{\ell-2},\widetilde\HH)\\
\nonumber &=\1\{\EE_{\ell}\cap\AA\cap\AA'\} \cdot  \P( G_{\ell+1}f^{(\ell-1)}=(G_{\ell+1}\varphi^{(\ell-2)})^{-1}| \FF_{\ell-2},\widetilde\HH)\\
&\leq\1\{\EE_{\ell}\cap\AA\cap\AA'\} \cdot  \max_{x\in G_{\ell}} \P( G_{\ell+1}f^{(\ell-1)}=G_{\ell+1}x| \FF_{\ell-2},\widetilde\HH).
\end{align}
To analyze the distribution of $G_{\ell+1} f^{(\ell-1)}$, we first use the fact that
$G_\ell/G_{\ell+1}$ is Abelian to write
\beq\label{f_sum}
G_{\ell+1} f^{(\ell-1)}
= \sum_{i=1}^{N+N'} G_{\ell+1}\bigl[s_{\sigma_i}^{\eta_i}, U_{i,\ell-1}\bigr].
\eeq
By Proposition~\ref{nilpotent_properties}(iii), we may remove the random signs
$\{\eta_i\}_{i\ge 1}$ by rewriting
\[
G_{\ell+1}\bigl[s_{\sigma_i}^{\eta_i}, U_{i,\ell-1}\bigr]
= G_{\ell+1}\bigl[s_{\sigma_i}, U_{i,\ell-1}^{\eta_i}\bigr]
= G_{\ell+1}\bigl[s_{\sigma_i}, U'_{i,\ell-1}\bigr],
\]
where $U'_{i,\ell-1}$ denotes the unique representative in $R_{\ell-1}$ of the coset
$G_\ell U_{i,\ell-1}^{\eta_i}$, which is again uniformly distributed on $R_{\ell-1}$. Grouping the terms in \eqref{f_sum} according to the values of $\sigma_i$, we may again rewrite \eqref{f_sum} as
\beq\label{f_sum_grouped}
G_{\ell+1} f^{(\ell-1)}
= \sum_{a=1}^k \sum_{j=1}^{N_a+N'_a}
G_{\ell+1}\bigl[s_a, U^{(a)}_{j,\ell-1}\bigr],
\eeq
where the collection $\{U^{(a)}_{j,\ell-1}\}_{j\ge 1,a\in[k]}$ consists of
i.i.d.\ uniform elements of $R_{\ell-1}$.

By Lemma~\ref{group_basics}(i), each random variable
$G_{\ell+1}[s_a, U^{(a)}_{j,\ell-1}]$ is i.i.d.\ uniform on the subgroup
$
\{ G_{\ell+1}[s_a, u] : u \in R_{\ell-1} \}.
$
Hence we have
\[
\sum_{j=1}^{N_a+N'_a} G_{\ell+1}\bigl[s_a, U^{(a)}_{j,\ell-1}\bigr]
\sim \Unif(H_{a,\ell})\quad \text{where }
H_{a,\ell}
:=
\begin{cases}
\{ G_{\ell+1}[s_a, u] : u \in R_{\ell-1} \}, & \text{if } N_a+N'_a \ge 1,\\[4pt]
\{ G_{\ell+1} \}, & \text{if } N_a+N'_a = 0 .
\end{cases}
\]
Applying Lemma~\ref{group_basics}(ii), we conclude that
$G_{\ell+1} f^{(\ell-1)}$, as the sum of independent uniform elements on the
subgroups $\{H_{a,\ell}\}_{a\in[k]}$, is uniform on
\[
H_\ell
:= \sum_{a\in[k]} H_{a,\ell}
= \{ h_1+\cdots+h_k : h_a \in H_{a,\ell} \}.
\]

Finally, on the event $\AA \cap \AA'$, we have $H_\ell = G_\ell / G_{\ell+1}$. Consequently, conditionally on $\sigma(\FF_{\ell-2}, \widetilde{\HH})$, $G_{\ell+1} f^{(\ell-1)}$ is uniformly distributed on $G_\ell / G_{\ell+1}$. It then follows from \eqref{conditional_prob} that,
\[
\P(\EE_{\ell+1} \cap \AA \cap \AA')
= \1\{\EE_\ell \cap \AA \cap \AA'\}\P(\EE_{\ell+1} \mid \FF_{\ell-2}, \widetilde{\HH})
\le \1\{\EE_\ell \cap \AA \cap \AA'\} \,|G_\ell / G_{\ell+1}|^{-1}.
\]
Taking expectations on both sides with respect to $\sigma(\FF_{\ell-2}, \widetilde{\HH})$ yields
\[
\P(\EE_{\ell+1}\cap \AA \cap \AA')
\le |G_\ell / G_{\ell+1}|^{-1}\,\P(\EE_\ell\cap \AA \cap \AA'), \qquad 2 \le \ell \le L.
\]
Iterating this bound over $\ell$ gives $\P( \EE_{L+1}\cap \AA\cap\AA')\leq  |G_2|^{-1}\P(\EE_2\cap \AA\cap\AA')$, and hence
\begin{align}\label{D_upperbound}
\nonumber D(t)&=|G|\cdot \P( \EE_{L+1} | \AA\cap\AA')-1\leq |G_{\ab}|\cdot \P(\EE_2| \AA\cap\AA')-1
\leq \frac{|G_{\ab}|\cdot \P(\EE_2)}{\P( \AA\cap\AA')}-1\\
\nonumber &=\frac{(d_{\ell^2}^\ab(t))^2+\P(( \AA\cap\AA')^c)}{\P( \AA\cap\AA')}\\
&\leq (1+2\P(( \AA\cap\AA')^c))\left((d_{\ell^2}^\ab(t))^2+\P(( \AA\cap\AA')^c)\right)
\end{align}
where $(d_{\ell^2}^\ab(t))^2=|G_{\ab}|\cdot \P(\EE_2)-1$ equals the $\ell^2$-distance to stationarity of the projected walk $X^\ab$ onto the abelianization $G_{\ab}$, and the last line follows from the the elementary inequality  $(1-u)^{-1}\leq 1+2u$ for $0\leq u\leq 1/2$ (assuming $\P(( \AA\cap\AA')^c)\leq 1/2$).

Let  $t=\max\{ t^{\ell^2}_{\mix}(G_{\ab},\ep/2), \mu_*^{-1}(\log k+2\log(4/\ep))\}$. As $\{N_a: a\in [k]\}$ are independent Poisson random variables with rate $\mu\bigl(\Cl(s_a)\bigr)$, a simple union bound yields 
$$\P(\AA^c)\leq \sum_{i=1}^k  e^{-\mu(\Cl(s_a))t}\leq ke^{-\mu_* t}\leq (\ep/4)^2$$
 and thus $\P((\AA\cap \AA')^c)\leq \ep^2/8.$ Combining \eqref{tv_triangle_inequality}, \eqref{L2_bound} and \eqref{D_upperbound} gives
\begin{align*}
\left\| \P(X_t=\cdot)-\pi \right\|_{\mathrm{TV}}\leq \frac{1}{2}\sqrt{D(t)}+\P(\AA^c)\leq \frac{1}{2} \sqrt{ (1+\ep^2/4)( (\ep/2)^2+\ep^2/8) }+\ep^2/16\leq \ep
\end{align*}
for any $\ep\in(0,1)$. We have therefore shown that 
$$
t_{\mix}^{\TV}(G,\varepsilon)
\leq\max\{ t^{\ell^2}_{\mix}(G_{\ab},\ep/2), \mu_*^{-1}(\log k+2\log(4/\ep))\},$$
and the proof is complete. 
\end{proof}

\subsection{Cutoff in Examples \eqref{superclass_walk} and \eqref{superclass_walk_nestoridi}}\label{sec:cutoff}

\begin{proof}[Proof of Corollary \ref{cutoff}]
We will apply Theorem \ref{reduction_abelianization} to prove matching upper and lower bounds of mixing times (up to lower order errors).
In both parts (a) and (b), the projected walks are product chains on
$\ZZ_p^{n-1}$ with independently evolving coordinates. Since the analyses are
entirely analogous, we focus on proving part (a) and only sketch the argument for part (b).

In part (a), the projected walk $X^\ab$ is a simple random walk on
$G_{\ab} \cong \ZZ_p^{n-1}$, where each coordinate evolves as an independent simple
random walk at rate $1/(n-1)$. Accordingly, we may represent
\[
X^{\ab}_t = \bigotimes_{i=1}^{n-1} Z_{i,s},
\]
where $\{Z_{i,\cdot}\}_{i=1}^{n-1}$ are independent rate-$1$ simple random walks on
$\ZZ_p$ and $s = t/(n-1)$.

To obtain the lower bound on the mixing time, we consider the discrete-time simple
random walk on $\ZZ_p$, whose transition matrix has eigenvalues
$\{\cos(2\pi j/p)\}_{j=0}^{p-1}$; see, for example, Section~12.3 of
\cite{wilmer2009markov}. In particular, the spectral gap is
$\gamma = 1 - \cos(2\pi/p)$. Applying Theorem~20.7 of \cite{wilmer2009markov} to
$X^\ab$, viewed as a rate-$1$ product chain on $\ZZ_p^{n-1}$, it follows
immediately that for any $\ep \in (0,1)$,
\beq\label{tv_lb}
t^{\TV}_{\mix}(G,1-\ep)\geq t_{\mix}^{\TV}(G_{\ab},1-\ep)\geq \frac{n-1}{2\gamma}\left(\log (n-1)-\log (8\log(1/\ep))\right).
\eeq

We proceed to upper bounding $t^{\TV}_{\mix}(G,\ep)$. Let $d_{\ell^2}^\ab(t)$ denote the $\ell^2$-distance to stationarity of $X^\ab$, and let $d^{(i)}(s)$ denote the $\ell^2$-distance to stationarity of
$Z_{i,s}$, where we set $s=t/(n-1)$. By independence among coordinates we can readily see that 
\begin{align}\label{Fourier_bound}
\nonumber d^\ab_{\ell^2}(t)^2+1&=p^{n-1}\P(X_t=X'_t)=\prod_{i=1}^{n-1} p\P(Z_{i, s}=Z'_{i,s})=\prod_{i=1}^{n-1} \left((d^{(i)}(s))^2+1\right)\\
 &\leq \exp\left( \sum_{i=1}^{n-1}(d^{(i)}(s))^2\right)= \exp\left( (n-1)\sum_{j=1}^{p-1} e^{-2(1-\cos(2\pi j/p))s}\right),
 \end{align}
where the last equality is a classical identity obtained by expanding $(d^{(i)}(s))^2$ through the Fourier representation of the walk on $\ZZ_p$. 

If $p=p_n$ is uniformly upper bounded over $n$ by some constant $M$, then we have
$$ (d^\ab_{\ell^2}(t))^2\leq \exp\left( (n-1)pe^{-(1-\cos(2\pi /p))2s}\right)-1\leq e^{\ep/4}-1\leq \ep/2$$
when $t\geq \frac{(n-1)\log (n-1)+(n-1)\log (8M/\ep)}{2(1-\cos(2\pi/p))}$. Note that in this case $\mu_*=1/(n-1)$ and hence Theorem \ref{reduction_abelianization}  implies that 
\begin{align*}
t^{\TV}_{\mix}(G,\ep)&\leq  \max\{ t_{\mix}^{\ell^2}(G_{\ab},\varepsilon/2), (n-1)\log (n-1)+2(n-1)\log(4/\ep)\}\\
&\leq  \frac{(n-1)\log (n-1)}{2(1-\cos(2\pi/p))}+(n-1)C_{\ep,M}
\end{align*}
for some constant $C_{\ep,M}>0$ dependent on $\ep,M$. Together with \eqref{tv_lb}, this completes the proof of cutoff in the case that $p=p_n$ is uniformly upper bounded.

Now suppose $p=p_n$ diverges as $n\to\infty$. For arbitrarily small $\delta>0$, when $x>0$ is sufficiently small one has $ 1-\cos x\geq (1-\delta) x^2/2$. Let $J=\{1,\dots, \floor{ \sqrt{p}}\}\cup \{ p-\floor{ \sqrt{p}}, \dots,p-1\}$. Then for any $j\in J$, we have $\widetilde j:=\min\{j,p-j\}\leq  \floor{ \sqrt{p}}$ and
$$1-\cos(2\pi j/p)=1-\cos(2\pi \widetilde j/p)\geq  (1-\delta)\frac{2\pi^2 \widetilde j^2}{p^2}\geq (1-\delta)\frac{2\pi^2 \widetilde j}{p^2}.$$

Hence,
\begin{align*}
\sum_{j=1}^{p-1} e^{-2(1-\cos(2\pi j/p))s}
&\le \sum_{j\in J} e^{-2(1-\cos(2\pi j/p))s}
   + \sum_{j\in [p-1]\setminus J} e^{-2(1-\cos(2\pi j/p))s}\\
&\le 2\sum_{m=1}^{\lfloor\sqrt p\rfloor} e^{-4(1-\delta)\pi^2 m s/p^2}
   + p\, e^{-2(1-\cos(2\pi\lfloor\sqrt p\rfloor/p))s}\\
&\le  4e^{-4(1-\delta)\pi^2 s/p^2} + p\,e^{-2\pi^2 s/p},
\end{align*}
where the last line follows from the geometric sum $\sum_{m=1}^{\infty} e^{-4(1-\delta)\pi^2 m s/p^2}\leq 2e^{-4(1-\delta)\pi^2 s/p^2}$ when $e^{-4(1-\delta)\pi^2 s/p^2}\le 1/2$. Note this is satisfied when  $t\geq \frac{(1+2\delta)p^2 (n-1)\log(n-1)}{4\pi^2}$ and $s=t/(n-1)$.

Let $\delta\in (0,1/4)$ be arbitrarily small. Plugging the above bound into \eqref{Fourier_bound} gives
\begin{align*}
d^\ab_{\ell^2}(t)^2
&\le \exp\!\left((n-1)\Bigl(4e^{-4(1-\delta)\pi^2 s/p^2}
      + p e^{-2\pi^2 s/p}\Bigr)\right)-1\\
&\le \exp\!\left(4(n-1)^{-\delta/2}
      + p(n-1)^{1-(1+2\delta)p/2}\right)-1\\
&\le 8(n-1)^{-\delta/2}
      + 2p(n-1)^{1-(1+2\delta)p/2}
= o_n(1),
\end{align*}
where in the last step we used that the exponent is $o_n(1)$ and the inequality
$e^u-1\le 2u$ for $u$ sufficiently small. Therefore, for any $\ep\in (0,1)$ and any $\delta\in (0,1/4)$, we have 
$$ t^{\TV}_{\mix}(G,\ep)\leq \max\{t_{\mix}^{\ell^2}(G_{\ab},\varepsilon/2),(n-1)\log (n-1)+2(n-1)\log(4/\ep)\}\leq \frac{(1+2\delta)p^2 (n-1)\log(n-1)}{4\pi^2}$$
when $n$ is sufficiently large.  Combined with \eqref{tv_lb}, this shows that the walk exhibits cutoff at time
$$
\frac{p^2(n-1)\log(n-1)}{4\pi^2},
$$
which is asymptotically equivalent to
$\frac{(n-1)\log(n-1)}{2(1-\cos(2\pi/p))}$ when $p=p_n\to\infty$.

The proof of part (b) follows analogously upon noting the eigenvalues of a simple random walk on $\ZZ_p$ with jumps $\{\pm 1, \pm b\}$ are $\{\frac{(\cos(2\pi j/p)+\cos(2\pi b j/p)}{2}\}_{j=0}^{p-1}$, where $b=\floor{\sqrt{a}}$.

\end{proof}

\bibliographystyle{plain}
\bibliography{conjugacy-invariant}
\end{document}